\documentclass[10pt,a4paper, xcolor]{article}
\usepackage{amsfonts}
\usepackage{amssymb,amsthm, amsmath,graphicx,bm,amsthm,mathtools }
\usepackage{mathrsfs}
\usepackage{color,enumerate,lineno}
\usepackage{url,hyperref}
\usepackage{color,enumerate}
\usepackage{url,lineno}
\usepackage[latin1]{inputenc}
\usepackage[all]{xy}

\newtheorem{theorem}{Theorem}[section]

\newtheorem{proposition}[theorem]{Proposition}
\newtheorem{corollary}[theorem]{Corollary}
\newtheorem{lemma}[theorem]{Lemma}
\newtheorem{algorithm}[theorem]{Algorithm}

\theoremstyle{remark}

\newtheorem{example}[theorem]{Example}
\newtheorem{remark}[theorem]{Remark}

\def\sA{\mathscr{A}}
\def\sB{{\mathscr{B}}}

\def\e{{\mathrm e}}
\def\g{{\mathrm g}}

\def\m{{\mathrm m}}

\def\t{{\mathrm t}}

\def\F{{\mathrm F}}
\def\G{{\mathrm G}}

\def\M{{\mathrm M}}
\def\rR{{\mathrm R}}

\def\Cad{{\mathrm {Cad}}}

\def\PF{{\mathrm {PF}}}

\def\Ap{{\mathrm{ Ap}}}
\def\max{{\mathrm{ max}}}
\def\MED{{\mathrm{MED}}}

\def\gcd{{\mathrm{gcd}}}

\def\max{{\mathrm{max}}}
\def\msg{{\mathrm{ msg }}}

\def\Maximals{\mathrm{Maximals}_{\leq_S}}

\def\CC{\mathscr{C}}

\def\N{\mathbb{N}}

\def\Z{\mathbb{Z}}

\def\rank{\mathrm{rank}\, }
\def\Ap{\mathrm{Ap}}
\def\SG{\mathrm{SG}}

\def\int{\mathrm{int}}

\title{The covariety of numerical semigroups with  fixed Frobenius number}

\author{
M. A. Moreno-Fr\'{\i}as \footnote{
    Dpto. de Matem\'aticas, Facultad de Ciencias,
    Universidad de C\'adiz, E-11510, Puerto Real  (C\'{a}diz, Spain).
    Partially supported by  Junta de Andaluc\'{\i}a group FQM-298 and by ProyExcel-00868.
     E-mail: mariangeles.moreno@uca.es.}
\and
 J. C. Rosales \footnote{
    Dpto. de \'Algebra, Facultad de Ciencias, Universidad de Granada,
    E-18071, Granada. (Spain).
    Partially supported by  Junta de Andaluc\'{\i}a group FQM-343 and by ProyExcel-00868.
    E-mail: jrosales@ugr.es.
}
 }

\date{}

\begin{document}

\maketitle

\begin{abstract}
	
	Denote by $\m(S)$ the multiplicity of a numerical semigroup $S$.  A {\it covariety}  is a nonempty family $\CC$ of numerical semigroups that  fulfills  the following conditions: there is the minimum of $\CC,$ the intersection of two elements of  $\CC$  is again an element of  $\CC$ and $S\backslash \{\m(S)\}\in  \CC$  for all $S\in  \CC$ such that  $S\neq \min(\CC).$ In this work we describe an algorithmic procedure to compute all the elements of $\CC.$ We prove that there exists the smallest element of $\CC$ containing a set of positive integers.  We show that  
	 $\sA(F)=\{S\mid S \mbox{ is a numerical semigroup with Frobenius number }F\}$ is a covariety, and we particularize   the previous results in this covariety. Finally, we will see that there is the smallest covariety containing a finite set of numerical semigroups.


    \smallskip
    {\small \emph{MSC-class:} 20M14, 11D07 }
\end{abstract}

\section{Introduction}

\hspace{0.42cm}Let $\Z$ be the set of integers and $\N=\{z\in \Z \mid z\ge 0\}$. A {\it submonoid} of $(\N,+)$ is a subset  of $\N$ which is closed under addition and contains the element $0.$ A {\it numerical semigroup} is a
submonoid $S$ of $(\N,+)$ such that $\N\backslash
S=\{x\in \N \mid x \notin S\}$ has finitely many elements.

If $S$ is a numerical semigroup, then $\m(S)=\min(S\backslash \{0\})$, $\F(S)=\max\{z\in \Z \mid z \notin S\}$ and $\g(S)=\sharp(\N \backslash S)$ (where $\sharp X $ denotes the cardinality of a set $X$) are 
three important invariants of $S$  and we call them the {\it multiplicity}, the {\it Frobenius number} and the {\it genus} of $S$, respectively.

If $A$ is a subset nonempty of $\N$, we denote by $\langle A
\rangle$ the submonoid of $(\N,+)$ generated by $A$, that is,
$\langle A \rangle=\{\lambda_1a_1+\dots+\lambda_na_n \mid n\in
\N\backslash \{0\}, \, \{a_1,\dots, a_n\}\subseteq A \mbox{ and }
\{\lambda_1,\dots,\lambda_n\}\subseteq \N\}.$  In \cite[Lema 2.1]{libro}) it is shown that $ \langle A \rangle$
is a numerical semigroup if and only if $\gcd(A)=1.$ 

If $M$ is a submonoid of $(\N,+)$ and $M=\langle A \rangle$, then we
say that $A$ is a {\it system of generators} of $M$. Moreover, if $M\neq
\langle B \rangle$ for all $B \varsubsetneq A$, then we will say
that $A$ is a {\it minimal system of generators} of $M$. In
\cite[Corollary 2.8]{libro} is shown that every submonoid of
$(\N,+)$ has a unique minimal system of generators, which in
addition is finite. We denote by $\msg(M)$ the minimal system of
generators of $M$. The cardinality of $\msg(M)$ is called the {\it
	embedding dimension }of $M$ and will be denoted by $\e(M).$

The Frobenius problem (see \cite{alfonsin})  focuses on finding formulas to calculate the Frobenius number and the genus of a numerical semigroup from its minimal  system of generators. The problem was solved in \cite{sylvester} for numerical semigroups with embedding dimension two.  Nowadays, the problem is still open in the case of numerical semigroups with embedding dimension  greater than or equal to three. Furthemore, in this case the problem of computing the Frobenius number of a general numerical semigroup becomes NP-hard.

If $F\in \N\backslash \{0\},$ then we consider the set $\sA(F)=\{S\mid S \mbox{ is a numerical semi-}\\\mbox{group and } \F(S)=F\}.$ The study of set  $\sA(F)$ has been the origin and motivation of this work.

The generalization of  $\sA(F)$ as a family of numerical semigroups that verifies certain properties leads us to give the following definition. A {\it covariety}  is a nonempty family $\CC$ of numerical semigroups that  fulfills  the following conditions:
\begin{enumerate}
	\item[1)]  $\CC$ has a minimum,  $\Delta(\CC)=\min(\CC)$ (with respect to set inclusion).
	\item[2)] If $\{S, T\} \subseteq \CC$, then $S \cap  T \in \CC$.
	\item[3)]  If $S \in \CC$  and $S \neq  \Delta(\CC)$, then $S \backslash \{\m(S)\} \in \CC$.
\end{enumerate}

The elements of $\CC$ are called $\CC${\it -semigroups}. 

In Section 2, we will see that every covariety is finite and its elements can be ordered making a tree. Moreover, we present a characterization of the children of an arbitrary vertex in this tree. This fact will allow us, in Section 3, to give a quite efficient algorithm to calculate all the elements of $\sA(F).$

If $\CC$ is a covariety, then it has a finite number $S_1,\dots,S_n$ of maximal elements. A $\CC$-{\it set} is a subset of $S_i\backslash \Delta(\CC)$ for some $i\in \{1,\dots, n\}.$ In Section 4, we will show that if $A$ is a $\CC$-set, then there is the smallest element of $\CC$ that contains $A.$ This element will be denoted by $\CC(A)$ and it will called the $\CC${\it -semigroup generated by }$A.$

If $S=\CC(A),$ then we will say that $A$ is a $\CC$-{\it system of generators} of $S.$ Furthemore, if $S\neq
\CC(B)$ for all $B \varsubsetneq A$, then we will say
that $A$ is a $\CC$-{\it minimal system of generators} of $S.$ In Section 4, we will show an example of covariety $\CC$ where the $\CC$-minimal system of generators is not unique. Also we will prove, in this section, that every $\sA(F)$-semigroup admits a unique $\sA(F)$-minimal system of generators. 

In Section 5, we will talk about the smallest covariety that contains a finite family of numerical semigroups, and we will give an algorithmic procedure to compute it.

If $\CC$ is a covariety and $S$ is a $\CC$-semigroup, then the $\CC$-{\it rank} of $\CC$  is $\CC_{\mbox{rank}}(S)=\min\{\ \sharp A\mid A \mbox{ is a }\CC\mbox{-set and }\CC(A)=S \}.$ In Section 6, we will see that $\CC_{\mbox{rank}}(S)\leq \e(S)$ and $\CC_{\mbox{rank}}(S)=0$ if and only if $S=\Delta(\CC).$

In Section 6, we will also see that if $S\in \sA(F)$, then $\sA(F)_{\mbox{rank}}(S)\leq \m(S)-1.$ We will study the $\sA(F)$-semigroups $S$ such that  $\sA(F)_{\mbox{rank}}(S)\in \{1, \m(S)-1\}$ and we will find formulas to calculate the genus of this kind of numerical semigroups, depending on its $\sA(F)$-minimal system of generators.

\section{The tree associated to a covariety}

\hspace{0.42cm} If $T$ is a numerical semigroup, then $\N\backslash T$ is a finite set and so we have the following result.
\begin{lemma}\label{lemma1}
If $T$ is a numerical semigroup, then the set $$\{ S \mid S \mbox{ is numerical semigroup and }T\subseteq S\}$$ is  finite.
\end{lemma}
\begin{proposition}\label{proposition2}
Every covariety has finite cardinality.	
\end{proposition}
\begin{proof}
	If $\CC$ is a covariety, then $\CC \subseteq \{S \mid S \mbox{ is a numerical semigroup and }\Delta(\CC)\subseteq S\}.$ By Lemma \ref{lemma1}, we know that $\{S \mid S \mbox{ is a numerical semigroup and }\Delta(\CC)\subseteq S\}$ is a finite set. Therefore $\CC$ is  finite.
\end{proof}

The following result is well known and it is easy to prove. 

\begin{lemma}\label{lemma3}
	Let $S$ and $T$ be numerical semigroups and $x\in S.$ Then
	\begin{enumerate}
		\item $S\cap T$ is also a numerical semigroup and $\F(S\cap T)=\max\{\F(S), \F(T)\}.$
		\item $S\backslash \{x\}$ is a numerical semigroup if and only if $x\in \msg(S).$
		\item $\m(S)=\min\left( \msg(S)\right).$
	\end{enumerate}
\end{lemma}

As a consequence of this lemma we have the following result.

\begin{proposition}\label{proposition4}
	If $F\in \N \backslash \{0\}$, then $\sA(F)$ is a covariety. Moreover, $$\Delta(\sA(F))=\{0,F+1,\rightarrow\},$$ where the symbol $\rightarrow$ means that every integer greater than $F+1$ belongs to the set.
\end{proposition}
A {\it graph} $G$ is a pair $(V,E)$ where $V$ is a nonempty set and
$E$ is a subset of $\{(u,v)\in V\times V \mid u\neq v\}$. The
elements of $V$ and $E$ are called {\it vertices} and {\it edges},
respectively. A {\it path (of
	length $n$)} connecting the vertices $x$ and $y$ of $G$ is a
sequence of different edges of the form $(v_0,v_1),
(v_1,v_2),\ldots,(v_{n-1},v_n)$ such that $v_0=x$ and $v_n=y$.

A graph $G$ is {\it a tree} if there exists a vertex $r$ (known as
{\it the root} of $G$) such that for any other vertex $x$ of $G$
there exists a unique path connecting $x$ and $r$. If  $(u,v)$ is an
edge of the tree $G$, we say that $u$ is a {\it child} of $v$.

Let $\CC$ be  a covariety and let $S\in \CC.$ Define recursively the following sequence of $\CC$-semigroups:
\begin{itemize}
	 \item $S_0=S$,
	 \item $S_{n+1} =\left\{\begin{array}{ll}
	 	S_n\backslash \{\m(S_n)\}  & \mbox{if }   S_n\neq \Delta(\CC),\\
	 	\Delta(\CC) & \mbox{otherwise.}
	 \end{array}
	 \right.$
\end{itemize}

The following result has an immediate proof. 

\begin{lemma}\label{lemma5}
	If $\CC$ is a covariety, $S\in \CC$ and $\{S_n\}_{n\in \N}$ is the  sequence of semigroups defined above, then there is $k\in \N$ such that $\Delta(\CC)=S_k\subsetneq S_{k-1} \subsetneq \dots \subsetneq S_0=S.$ Moreover, the cardinality of $S_i\backslash S_{i+1}$ is equal to $1$ for every $i\in \{0,1,\dots,k-1\}.$ 	
\end{lemma}

If $\CC$ is a covariety, we define the graph $\G(\CC)$ in the following way:
\begin{itemize}
\item the set of vertices of $\G(\CC)$ is $\CC$,
\item $(S,T)\in \CC \times \CC$ is an edge of $\G(\CC)$ if and only if $T=S\backslash \{\m(S)\}.$
	\end{itemize}

As a consequence of Lemma \ref{lemma5}, we have the following result.
\begin{proposition}\label{proposition6}
	If $\CC$ is a covariety, then $\G(\CC)$ is a tree with root $\Delta(\CC).$	
\end{proposition}

A tree can be  built recurrently starting from the root  and connecting, 
through an edge, the vertices already built with  their  children. Hence, it is very interesting to characterize the children of an arbitrary vertex in a tree.

Following the terminology introduced in \cite{JPAA}, an integer $z$ is a {\it pseudo-Frobenius number} of $S$ if $z\notin S$ and $z+s\in S$ for all $s\in S\backslash \{0\}.$  We denote by $\PF(S)$
the set of pseudo-Frobenius numbers of $S.$ The cardinality of $\PF(S)$ is an important invariant of $S$ (see \cite{froberg} and \cite{barucci}) called the {\it type} of $S,$  denoted by $\t(S).$

The following result is Corollary 2.23 from \cite{libro}. 

\begin{lemma}\label{lemma7}If $S$ is a numerical semigroup and $S\neq \N,$ then $\t(S)\leq \m(S)-1.$
\end{lemma}

Given a numerical semigroup $S,$ denote by $\SG(S)=\{x\in \PF(S)\mid 2x\in S\}.$
The elements of $\SG(S)$ will be called the {\it special gaps} of $S.$ The following result is Proposition 4.33 from \cite{libro}.

\begin{lemma}\label{lemma8}
	Let $S$ be a numerical semigroup and $x\in \N\backslash S.$ Then $x\in \SG(S)$ if and only if $S\cup \{x\}$ is a numerical semigroup.
\end{lemma}
\begin{proposition}\label{proposition9} If $\CC$ is a covariety and $S\in \CC,$ then the set formed by the children of $S$ in the tree $\G(\CC)$ is $\{S\cup \{x\}\mid x\in \SG(S), x<\m(S) \mbox { and }S\cup \{x\}\in \CC\}.$
\end{proposition}
\begin{proof}
	If $T$ is a child of $S,$ then $T\in \CC$ and $T\backslash \{\m(T)\}=S.$ Therefore, $S\cup \{\m(T)\}=T\in \CC,$ $\m(T) \in \SG(S)$ and $\m(T)<\m(S).$
	
	If $x<\m(S),$ then $\m(S\cup \{x\})=x.$ Hence, $S\cup \{x\}\in \CC$  and $\left( S\cup \{x\}\right)\backslash \{\m(S\cup \{x\})\}=S.$ Consequently, $S\cup \{x\}$ is a child of $S$ in the tree $\G(\CC).$
\end{proof}

\section{Algorithm to compute $\sA(F)$}

\hspace{0.42cm}Let $S$ be a numerical semigroup and $n\in S\backslash \{0\}$. The
Apéry set of $n$ in $S$ (named so in honour of \cite{apery}) is
$\Ap(S,n)=\{s\in S\mid s-n \notin S\}$. 

The following result is deduced from \cite[Lemma 2.4]{libro}.

\begin{lemma}\label{lemma10}
	If $S$ is a  numerical semigroup and $n\in S\backslash \{0\},$ Then $\Ap(S,n)$ is a set with cardinality $n.$ Moreover, $\Ap(S,n)=\{0=w(0),w(1), \dots, w(n-1)\}$, where $w(i)$ is the least
	element of $S$ congruent with $i$ modulo $n$, for all $i\in
	\{0,\dots, n-1\}.$
\end{lemma}


If $S$ is a numerical semigroup, then we define over $\Z$ the following order relation: $a\leq_S b$ if and only if $b-a\in S.$ The following result is \cite[Lemma 10]{JPAA}.
\begin{lemma}\label{lemma11} If $S$ is  a numerical semigroup and $n \in S\setminus
	\{0\}.$ Then
	$$
	\PF(S)=\{w-n\mid w \in \Maximals \Ap(S,n)\}.
	$$
\end{lemma}
The following result has an immediate proof.

\begin{lemma}\label{lemma12}
	Let $S$ be a numerical semigroup, $n\in S\backslash \{0\}$ and $w \in \Ap(S,n).$ Then $w\in \Maximals(\Ap(S,n))$ if and only if $w+w'\notin
	\Ap(S,n)$ for all $w'\in \Ap(S,n)\backslash \{0\}. $\\	
\end{lemma}

The following result has a simple proof.

\begin{lemma}\label{lemma13} If $S$ is a numerical semigroup and $S\neq \N,$ then $$\SG(S)=\{x\in \PF(S)\mid 2x \notin \PF(S)\}.$$
\end{lemma}


\begin{remark}\label{remark14}
	Note that if $S$ is a numerical semigroup and we know $\Ap(S,n)$ for some $n\in S \backslash \{0\},$ as a consequence of Lemmas \ref{lemma11}, \ref{lemma12} and \ref{lemma13}, we can easily compute $\SG(S).$
\end{remark}
We will illustrate the content of the previous remark with an example. 
\begin{example}\label{example15} Let $S=\{0,5,\rightarrow\}.$ Then $\Ap(S,5)=\{0,6,7,8,9\}.$ By applying Lemma \ref{lemma12}, we have that $\Maximals \Ap(S,5)=\{6,7,8,9\}.$ Then Lemma \ref{lemma11} asserts that $\PF(S)=\{1,2,3,4\}.$ Finally, by using Lemma \ref{lemma13}, we obtain $\SG(S)=\{3,4\}.$
	
\end{example}

The following result is easy to prove.

\begin{lemma}\label{16}
	Let $S$ be a numerical semigroup, $n\in S \backslash \{0\}$ and $x\in \SG(S).$ Then $x+n\in \Ap(S,n).$ Furthemore, $\Ap\left( S\cup \{x\},n \right)=\left( \Ap(S,n)\backslash \{x+n\}\right)\cup \{x\}.$
\end{lemma}

\begin{remark}\label{17}
	Observe that as a consequence from Lemma \ref{16}, if we know $\Ap(S,n),$ then we can easily compute $\Ap(S\cup  \{x\},n).$ In particular, if $\CC$ is a covariety and $S\in \CC,$ then by Lemma \ref{16} allows to calculate the set $\Ap(T,n)$ from $\Ap(S,n),$ for all child $T$ of $S$ in the tree $\G(\CC)$(see Proposition \ref{proposition9}).
\end{remark}
	
	Next we illustrate this remark with an example.
	\begin{example}\label{example18}
		We take again the numerical semigroup $S=\{0,5,\rightarrow\}$ with   $\Ap(S,5)=\{0,6,7,8,9\}.$ By Example \ref{example15} we know that $3\in \SG(S).$ If $T=S\cup \{3\},$ then the Lemma \ref{16}, asserts that $\Ap(T,5)=\left(\{0,6,7,8,9\}\backslash \{3+5\}\right)\cup \{3\}=\{0,3,6,7,9\}.$
	\end{example}

As a consequence of Propositions \ref{proposition4} and \ref{proposition9}, we have the following result. 

\begin{proposition}\label{proposition19}
	Let $F\in \N\backslash \{0\}$ and $S\in \sA(F).$ Then the set formed by the children of $S$ in the tree $\G(\sA(F))$ is $\left\{ S\cup \{x\}\mid x \in \SG(S), \, x<\m(S) \mbox{and } x\neq F\right\}.$	
\end{proposition}

We are now ready to show the  algorithm which gives title to this section. 

The pseudo-code in Algorithm \ref{algorithm20} shows how to compute $\sA(F)$ for any positive integer $F.$

\begin{algorithm}\label{algorithm20}\mbox{}\par
\end{algorithm}
\noindent\textsc{Input}: A positive integer $F.$   \par
\noindent\textsc{Output}: $\sA(F).$

\begin{enumerate}
	\item[(1)] $\sA(F)=\{\Delta(\sA(F))\},$ $B=\{\Delta(\sA(F))\}$ and $\Ap(\Delta(\sA(F)),F+1)=\{0,F+2, \dots, 2F+1\}.$ 
	\item[(2)] For every $S \in B,$ by using Remark \ref{remark14}, compute $$\theta(S)=\{x\in \SG(S)\mid x<\m(S) \mbox{ and }x\neq F\}.$$
	\item[(3)] If $\displaystyle\bigcup_{S\in B}\theta(S)=\emptyset,$ then return $\sA(F).$
	\item[(4)]  $C=\displaystyle\bigcup_{S\in B}\{S\cup \{x\}\mid x\in \theta(S)\}.$ 	
	\item[(5)]  $\sA(F)= \sA(F)\cup C,$ $B=C,$ by using Remark \ref{17}, compute $\Ap(S,F+1)$ for all $S\in C$ and   go to Step $(2).$ 	
\end{enumerate}

In the following example we illustrate the usage of Algorithm \ref{algorithm20}.

\begin{example}
	We are going to compute $\sA(5).$ In order to simplify notation we will 
	write $\Delta$ to denote $\Delta(\sA(5)).$
	\begin{itemize}
		\item $\sA(5)=\{\Delta=\{0,6,\rightarrow\}\},$ $B=\{\Delta\}$ and $\Ap(\Delta,6)=\{0,7,8,9,10,11\}.$
		\item $\theta(\Delta)=\{3,4\}$ and $C=\{\Delta \cup \{3\},\Delta \cup \{4\}\}.$
		\item $\sA(5)=\{\Delta, \Delta \cup \{3\},\Delta \cup \{4\}\}, $ $B=\{\Delta \cup \{3\},\Delta \cup \{4\}\},$ $\Ap(\Delta \cup \{3\},6)=\{0,3,7,8,10,11\}$ and $\Ap(\Delta \cup \{4\},6)=\{0,4,7,8,9,11\}.$
		\item $\theta(\Delta \cup \{3\})=\emptyset,$ $\theta(\Delta \cup \{4\})=\{2,3\}$ and  $C=\{\Delta \cup \{2,4\},\Delta \cup \{3,4\}\}.$
		\item $\sA(5)=\{\Delta, \Delta \cup \{3\},\Delta \cup \{4\}, \Delta \cup \{2,4\}, \Delta \cup \{3,4\} \}, $ $B=\{\Delta \cup \{2,4\},\Delta \cup \{3,4\}\},$ $\Ap(\Delta \cup \{2,4\},6)=\{0,2,4,7,9,11\}$ and $\Ap(\Delta \cup \{3,4\},6)=\{0,3,4,7,8,11\}.$
		\item $\theta(\Delta \cup \{2,4\})=\emptyset$ and $\theta(\Delta \cup \{3,4\})=\emptyset.$ Therefore the Algorithm return $\sA(5)=\{\Delta, \Delta \cup \{3\}, \Delta \cup \{4\},\Delta \cup \{2,4\},\Delta \cup \{3,4\}\}.$
		
		\end{itemize}
\end{example}
\section{The least $\CC$-semigroup that contains a $\CC$-set}
\hspace{0.42cm}Along this section $\CC$ denotes a covariety and $\{M_1,\dots, M_n\}$ is the set formed by the maximal elements of $\CC,$ with respect to set inclusion. Recall that a $\CC$-set is a subset of $M_i\backslash \Delta(\CC)$ for some $i\in \{1,\dots,n\}.$

If $A$ is a $\CC$-set, then we denote by $\CC(A)$ the intersection of all 
 the $\CC$-semigroups containing $A.$ As $\CC$ is finite, then the intersection of $\CC$-semigroups is also a $\CC$-semigroup. Therefore, we can state the following result. 
 
 \begin{proposition}\label{proposition22}
 	If $A$ is a $\CC$-set, then $\CC(A)$ is the least (with respect to set inclusion) $\CC$-semigroup containing $A.$
 	
 \end{proposition}

\begin{proposition}\label{proposition23} 
	If $S\in \CC,$ then $A=\{x\in \msg(S)\mid x\notin \Delta(\CC)\}$ is a $\CC$-set and $\CC(A)=S.$
\end{proposition}
\begin{proof}
	It is clear that $A$ is a $\CC$-set. As $S\in \CC$ and $A\subseteq S,$ then $\CC(A)\subseteq S.$ If $T\in \CC$ and $A\subseteq T,$ then $A\cup \Delta(\CC)\subseteq T.$ Therefore $\msg(S)\subseteq T$ and consequently $S\subseteq T.$ Thus $S\subseteq \CC(A).$ We conclude that $S=\CC(A).$
\end{proof}

As a consequence of Proposition \ref{proposition23}, we have the following result.
\begin{corollary}\label{corollary24} Under the standing notation, $\CC=\{\CC(A)\mid A \mbox{ is a }\CC\mbox{-set}\}.$
	
\end{corollary}

If $A$ is a $\CC$-set and $S=\CC(A)$, then we say that $A$ is a $\CC$-{\it system of generators} of $S.$ Furthemore, if $S\neq
\CC(B)$ for all $B \varsubsetneq A$, then we will say
that $A$ is a $\CC$-{\it minimal system of generators} of $S.$

In general, the $\CC$-minimal system of generators are not unique  as the following example  shows.
\begin{example}\label{example25}
	
	Let $\CC=\{S \mid S \mbox{ is a numerical semigroup, }\m(S)\ge 6 \mbox{ and }\F(S)\leq 11\}\cup \left\{ S_1=\{0,5,6,7,10, \rightarrow\},  S_2=\{0,5,8,9,10, \rightarrow\},  S_3=\{0,5,10, \rightarrow\}\right\}.$
	\begin{itemize}
		\item It is clear that $\Delta(\CC)=\{0,12,\rightarrow\}$ is the minimum of $\CC.$
		\item We can see straightforwardly that the intersection of two elements of $\CC$ is again an element of $\CC.$
		\item It is clear that if $S\in \CC$ and $S\neq \Delta(\CC),$ then $S\backslash \{\m(S)\}\in \CC.$
		\item As a consequence  of the three previous claims, we have that $\CC$ is a covariety.
		\item It is easy to verify that $\CC(\{5,6\})=\CC(\{5,7\})=S_1,$ $\CC(\{5\})=S_3,$ $\CC(\{6\})=\{0,6,12,\rightarrow\}),$ $\CC(\{7\})=\{0,7,12,\rightarrow\})$ and $\CC(\{ \emptyset\})=\{0,12,\rightarrow\}.$
	\end{itemize}
	Therefore, the sets $\{5,6\}$ and $\{5,7\}$ are  $\CC$-minimal system of generators of $S_1$	
\end{example}

Finally, note that  the above example also shows that, in general,  the set $\{x\in \msg(S)\mid x \notin \Delta(\CC)\}$ is not a minimal system of generators of $S.$ In fact, $\{x\in \msg(S_1)\mid x \notin \{0,12, \rightarrow\}\}=\{5,6,7\}$ and $\CC(\{5,7\})=S_1.$

Our next aim in this section will be to prove that every $\sA(F)$-semigroup admits a unique $\sA(F)$-minimal system of generators.

%
%
%
%
\begin{proposition}\label{proposition28}
	If $F\in \N \backslash \{0\}$ and $S \in \sA(F),$ then the set $A=\{x\in \msg(S)\mid x\notin \Delta(\sA(F))\}$ is the unique $\sA(F)$-minimal system of generators of $S.$
	\end{proposition}
\begin{proof}
	By Proposition \ref{proposition23}, we know that  $A$ is an $\sA(F)$-set and $\sA(F)(A)=S.$ 
	To conclude the proof, we will see that if $B$ is an $\sA(F)$-set and $\sA(F)(B)=S,$ then $A \subseteq B.$ In fact, if $x\in A\backslash B,$ then $S\backslash \{x\}$ is an $\sA(F)$-semigroup and $B\subseteq S\backslash \{x\}.$ Therefore, $\sA(F)(B)\subseteq S\backslash \{x\},$ which is impossible.
\end{proof}

We end this section by characterizing the maximal elements of the covariety $\sA(F).$

 A numerical semigroup is {\it irreducible} if 
it can not be expressed as an intersection of two numerical semigroups containing it properly. This concept was introduced in  \cite{irreducibles} where the following result  is also proven.
\begin{proposition}\label{proposition29}Let $S$ be a numerical semigroup. Then $S$ is irreducible if and only if $S$ is a maximal element in $\sA(\F(S)).$	
\end{proposition}

The irreducible numerical semigroups have great interest because from \cite{barucci} and \cite{froberg}, can be deduced that a numerical semigroup $S$ is irreducible if and only if $S$ is a	symmetric  or  pseudo-symmetric numerical semigroup. These kind of semigroups has been has been widely treated in the literature because a one dimensional local domain analytically irreducible is  Gorenstein (respectively Kunz) if and only if its value semigroup is symmetric (respectively pseudo-symmetric), see \cite{kunz} and \cite{barucci}. Finally, we will mention that in \cite{forum} appears an algorithm which computes all the irreducible numerical semigroups with a fixed Frobenius number. These results are used in \cite{blanco} to give and algorithm that allows to compute all the elements of $\sA(F).$

\section{The covariety generated by a finite family of numerical semigroups}

\hspace{0.42cm}In general, the intersection of covarieties is not a covariey. In fact, if $S$ and $T$ are numerical semigroups with $S\neq T,$ then $\CC_1=\{S\}$ and  $\CC_1=\{T\}$ are covarieties an  $\CC_1\cap  \CC_2=\emptyset$ which is not a covariety. 

The following result has an immediate proof.

\begin{lemma}\label{lemma30} Let $\{\CC_i\}_{i\in I}$ be a family of covarieties such that $\Delta(\CC_i)=\Delta$ for all $i\in I,$ then $\bigcap_{i\in I}\CC_i$ is a covariety and $\Delta$ is its minimum element.	
\end{lemma}

If $F\in \N\backslash \{0\},$ then we denote by $$\sB(F)=\{S\mid S \mbox{ is a numerical semigroup and } \F(S)\leq F\}.$$

The following result is straightforward to prove.

\begin{lemma}\label{lemma31}If $F\in \N\backslash \{0\},$ then $\sB(F)$
	 is a covariety and its minimum is $\Delta(\sB(F))=\{0,F+1,\rightarrow\}.$
\end{lemma}
The following result has an immediate proof.
\begin{lemma}\label{lemma32}
	If $S_1,\dots, S_n$ are numerical semigroups and $F=\max\{\F(S_1),\dots,\\
	 \F(S_n)\},$ then $\{S_1,\dots, S_n\}\subseteq \sB(F).$	
\end{lemma}
If $\{S_1,\dots, S_n\}$ is a finite set of numerical semigroups and  $F=\max\{\F(S_1),\dots,\\
 \F(S_n)\},$ then we denote by $\displaystyle \langle  \{S_1,\dots,S_n\}\rangle,$ the
intersection of all covarieties that contain $\{S_1,\dots, S_n\}$ and have $\{0,F+1, \rightarrow\}$ as minimum. 

As a consequence of Lemmas \ref{lemma30}, \ref{lemma31} and \ref{lemma32}, we have the following result. 

\begin{proposition}\label{proposition33} If $S_1,\dots, S_n$ are numerical semigroups, then $\displaystyle \langle  \{S_1,\dots,S_n\}\rangle$ is the smallest (with respect to set inclusion) covariety that contains $\{S_1,\dots, S_n\}$ and has $\{0,F+1, \rightarrow\}$ as minimum. 	
\end{proposition}

The covariety 
$\displaystyle \langle  \{S_1,\dots,S_n\}\rangle$ is called the {\it covariety generated }by $\{S_1,\dots, S_n\}.$ Our main aim will be to give an algorithmic procedure to
compute all elements of $\displaystyle \langle  \{S_1,\dots,S_n\}\rangle$ from $\{S_1,\dots, S_n\}.$

For all $i \in \{1,\dots, n\},$   we recurrentely define the following sequence:

\begin{itemize}
	\item $S^0_i = S_i$,
	\item $S^{n+1}_i =\left\{\begin{array}{lcl}
		S^n_i\backslash \{\m(S_i^n)\}  & &\mbox{ if }\,\, S_i^n\neq \{0,F+1,\rightarrow\},\\
		\{0,F+1,\rightarrow\}  & &\mbox{ otherwise. }\\
	\end{array}
	\right.$
\end{itemize}

The  following result has an immediate proof.

\begin{lemma}\label{lemma34}For every $i\in \{1,\dots,n\}$ there exists $P_i=\min\{k\in \N\mid S_i^k=\{0,F+1,\rightarrow\}\}.$
	\end{lemma}

For all $i\in \{1,\dots,n\},$ we will denote by 	$\Cad(S_i)=\{S_i^0,\dots, S_i^{P_i}\}.$ Note that $\{0,F+1,\rightarrow\}=S_i^{P_i}\subsetneq S_i^{P_i-1}\subsetneq \dots \subsetneq S_i^0=S_i$ and $\sharp (S_i^k\backslash S_i^{k+1})=1$ for every $k \in \{0,\dots, P_i-1\}.$

The following proposition  provides us the previously announced result.

\begin{proposition}\label{proposition35}
	If $S_1,\dots, S_n$ are numerical semigroups, then $\displaystyle \langle  \{S_1,\dots,S_n\}\rangle=\{\bigcap_{b\in B}T_b\mid \emptyset \neq B \subseteq \{1,\dots,n\} \mbox{ and } T_b\in \Cad(S_b) \mbox{ for all }b\in B\}.$
	\end{proposition}
\begin{proof}
	To prove this proposition it is enough to see that $\CC=\{\bigcap_{b\in B}T_b\mid \emptyset \neq B \subseteq \{1,\dots,n\} \mbox{ and } T_b\in \Cad(S_b) \mbox{ for all }b\in B\}$ is a covariety.
	\begin{itemize}
		\item We can see  that $\{0,F+1,\rightarrow\}$ is the minimum of $\CC.$
		\item It is clear that the intersection  of two elements of $\Cad(S_i)$ is again an element of $\Cad(S_i).$
		\item As a consequence of the previous set, we easily obtain that the intersection of two elements of $\CC$ is again an element of $\CC.$
		\item Now, we will see that if $T\in \CC$ and $T\neq \{0,F+1,\rightarrow\},$ then $T\backslash \{\m(T)\}\in \CC.$
		
		Indeed, if $T\in \CC,$ then there exists $\emptyset \neq B \subseteq \{1,\dots,n\}$ and  there is $ T_b\in \Cad(S_b)$ for all $b\in B$ such that $T=\bigcap_{b\in B}T_b.$ As $\m(T)\in T,$ then $\m(T)\in S_b$ for all $b\in B.$ For every $b\in B,$ we denote by $T'_b=\{x\in T_b\mid x>\m(T)\}\cup \{0\}.$ It is clear that $T'_b\in \Cad(S_b)$ for all $b\in B$ and  $T\backslash \{\m(T)\}=\bigcap_{b\in B}T'_b.$ Therefore, $T\backslash \{\m(T)\}\in \CC.$
		\item As a consequence of the previous points, we have that $\CC$ is a covariety.
	\end{itemize}
\end{proof}
The following result is an immediate consequence of Proposition \ref{proposition35}.
\begin{corollary} \label{corollary36}
	If $S$ is a numerical semigroup, then $\langle \{S\}\rangle=\Cad(S).$
\end{corollary}

We end this section by giving an example that illustrates the content of  Proposition \ref{proposition35}.
\begin{example}\label{example37}
Let $S_1=\langle 5,7,9\rangle=\{0,5,7,9,10,12,14,\rightarrow\}$ and $S_2=\langle 4,6,9\rangle=\{0,4,6,8,9,10,12,\rightarrow\}.$ Then $13=\max\{\F(S_1)=13,\F(S_2)=11\},$ $\Cad(S_1)=\displaystyle \left\{S_1, S_1\backslash \{5\}, S_1\backslash \{5,7\}, S_1\backslash \{5,7,9\}, S_1\backslash \{5,7,9,10\}, S_1\backslash \{5,7,9,10,12\}\right\}$ and \\ $\Cad(S_2)=\displaystyle \left\{S_2, S_2\backslash \{4\}, S_2\backslash \{4,6\}, S_2\backslash \{4,6,8\}, S_2\backslash \{4,6,8,9\}, S_2\backslash \{4,6,8,9,10\},\right.$\\ $\left\{ S_2\backslash \{4,6,8,9,10,12\},
 S_2\backslash \{4,6,8,9,10,12,13\} \right\}.$

By applying Proposition \ref{proposition35}, we obtain that
$\displaystyle \langle  \{S_1,S_2\}\rangle=\Cad(S_1)\cup \Cad(S_2)\cup \{T_1\cap T_2 \mid T_1\in \Cad(S_1) \mbox{ and } T_2\in \Cad(S_2)\}.$

\end{example}
\section{The $\CC$-rank of a $\CC$-semigroup}
\hspace{0.42cm}If $\CC$ is a covariey and $S$ is a $\CC$-semigroup, then the $\CC$-{\it rank} of $S$ is $\CC_{\rank}(S)=\min\{\sharp A\mid A \mbox{ is a }\CC\mbox{-set and } \CC(A)=S\}.$

As a consequence of Proposition \ref{proposition23}, we have the following result.
\begin{proposition}\label{proposition38} If $\CC$ is a covariety and $S$ is a $\CC$-semigroup, then $\CC_{\rank}(S)\leq \e(S).$
\end{proposition}
The following result has an immediate proof.
\begin{proposition}\label{proposition39} Let $\CC$ be a covariety and $S\in \CC.$ Then $\CC_{\rank}(S)=0$ if and only if $S=\Delta(\CC).$	
\end{proposition}
\begin{lemma}\label{lemma40} Let $\CC$ be a covariety and $S\in \CC$ such that $S\neq \Delta(\CC).$ If $A$ is a $\CC$-set and $\CC(A)=S,$ then $\m(S)\in A.$	
\end{lemma}
\begin{proof}
	If $\m(S)\notin A,$ then $S\backslash \{\m(S)\}$ is a $\CC$-semigroup that contains $A.$ Therefore, $\CC(A)\subseteq S\backslash \{\m(S)\},$ which is absurd. 
\end{proof}
\begin{proposition}\label{proposition41} Let $\CC$ be a covariety and $S\in \CC$ such that $S\neq \Delta(\CC).$ Then the following  conditions are equivalent.
	\begin{enumerate}
		\item $\CC_{\rank}(S)=1.$
		\item There is $a\in S$ such that $S=\CC(\{a\}).$
		\item $S=\CC(\{\m(S)\}).$
	\end{enumerate}
\end{proposition}
\begin{proof}
	{\it 1) implies 2).} Trivial.
	
	{\it 2) implies 3).} It is a consequence of Lemma \ref{lemma40}.
	
		{\it 3) implies 1).} It follows from Proposition \ref{proposition39}.
\end{proof}

As a consequence of Proposition \ref{proposition28}, we have the following result.
\begin{proposition}\label{proposition42}
	Let $\{a_1,\dots, a_n,a_{n+1}=F\}\subseteq \N$ where  $0<a_1<\dots<a_n<a_{n+1}$ and $a_{i+1}\notin  \displaystyle \langle  \{a_1,\dots,a_i\}\rangle$ for all $i\in \{1,\dots, n\}.$ Then $\displaystyle \langle  \{a_1,\dots,a_n\}\rangle \cup \{F+1,\rightarrow\}$ is an $\sA(F)$-semigroup with $\sA(F)$-$\rank$  equal to $n.$  Moreover, every  $\sA(F)$-semigroup with $\sA(F)$-$\rank$ equal to $n$ has this form.	
\end{proposition}

For integers $a$ and $b,$ we say that $a$ {\it divides} $b$ if there exists an integer $c$ such that $b=ca,$ and we denote this by $a\mid b.$ Otherwise, $a$ {\it does not divide} $b$, and we denote this by $a\nmid b.$

As an immediate consequence of Proposition \ref{proposition42}, we have the following result.
\begin{corollary}\label{corollary43} If $\{m,F\}\subseteq \N$ such that $0<m<F$ and $m\nmid F,$ then $\langle m \rangle \cup \{F+1,\rightarrow\}$ is an $\sA(F)$-semigroup with $\sA(F)$-$\rank$ equal to one.  Moreover, every  $\sA(F)$-semigroup with $\sA(F)$-$\rank$ equal to one has this form.	
\end{corollary}

	If $q$ is a rational number,  $\lfloor q \rfloor=\max \{z\in \Z\mid z\leq q\}.$ 
	
\begin{proposition}\label{proposition44}
	If $F\in \N\backslash \{0\},$ $S\in \sA(F)$ and $\sA(F)_{\rank(S)}=1,$ then $\g(S)=F-\displaystyle \left\lfloor \frac{F}{\m(S)}\right\rfloor.$	
\end{proposition}
\begin{proof}
	From Corollary \ref{corollary43}, we deduce that $S=\langle \m(S) \rangle \cup \{F+1,\rightarrow\}.$ Therefore, $S=\left\{0,\m(S),2\m(S),\dots,\left\lfloor \frac{F}{\m(S)}\right\rfloor \m(S) \right\} \cup \{F+1,\rightarrow\}.$ Hence, $\g(S)=F-\displaystyle \left\lfloor \frac{F}{\m(S)}\right\rfloor.$
\end{proof}
	\begin{example}\label{example45} Let $\sA(15)(\{6\}).$ By applying Proposition \ref{proposition44}, we have that $\g(S)=15-\displaystyle \left\lfloor \frac{15}{6}\right\rfloor=15-2=13.$
		\end{example} 
	
	Let $P_1,\dots, P_r$ be positive prime intergers and $\{\alpha_1,\dots, \alpha_r\}\subseteq \N.$ We know that the number of positive divisors of $P_1^{\alpha_1}\dots P_r^{\alpha_r}$ is $(\alpha_1+1)\cdots (\alpha_r+1).$  Then, as a consequence of Corollary \ref{corollary43}, we have the following result. 
	
	\begin{proposition}\label{proposition46} Let $F$ be an integer such that $F\ge 2.$ If $F=P_1^{\alpha_1}\dots P_r^{\alpha_r}$ is the decomposition of $F$ into primes, then the set $\{S\in \sA(F)\mid  \sA(F)_{\rank(S)}=1\}$ has cardinality $F-(\alpha_1+1)\cdots (\alpha_r+1).$	
	\end{proposition}
\begin{example}\label{example47}
	As $72=2^3\cdot3^2$, then by applying Proposition \ref{proposition46}, the set  $\{S\in \sA(72)\mid  \sA(72)_{\rank(S)}=1\}$ has cardinality $72-(3+1)\cdot(2+1)=72-12=60.$	
\end{example}

The next aim  is to study  the $\sA(F)$-semigroups with maximum rank. For this, we need to introduce some concepts and results.

The following result is deduced from \cite[Proposition 2.10]{libro}.

\begin{lemma}\label{lemma48} Let $S$ be a numerical semigroup. Then $\e(S)\leq \m(S).$
\end{lemma}	

A numerical semigroup $S$ is  said
to be   {\it maximal
embedding dimension}(referred to henceforth as $\MED$-semigroup) if $\e(S) = \m(S).$

In the literature one can find a long list of works dealing with the study
of one dimensional analytically irreducible local domains via their value
semigroups. One of the properties
studied for this kind of rings using this approach is that of being of maximal
embedding dimension (see [ \cite{abhyankar}, \cite{barucci}, \cite{brown-herzog} and \cite{sally}]).

The following result is deduced from \cite[Corollary 3.2]{libro}.

\begin{lemma}\label{lemma49}If $S$ is a $\MED$-semigroup, then $\F(S)=\max\left(\msg(S) \right)-\m(S).$
\end{lemma}
\begin{proposition}\label{proposition50} Let $F\in \N\backslash \{0\}$ and $S\in \sA(F).$ Then $\sA(F)_{\rank}(S)\leq \m(S)-1.$
\end{proposition}
\begin{proof}
	By Proposition \ref{proposition38}, we know that $\sA(F)_{\rank(S)}\leq \e(S)$ and by Lemma \ref{lemma48}, we have  $\e(S) \leq \m(S).$ If $\e(S) =\m(S),$ then by applying Lemma \ref{lemma49}, we deduce that $F<\max(\msg(S)).$ By using Proposition \ref{proposition28}, we conclude that $\sA(F)_{\rank}(S)\leq \e(S)-1= \m(S)-1.$
\end{proof}
If $S\in \sA(F)$ and  $\sA(F)_{\rank}(S)=\m(S)-1,$ then we say that $S$ is an  $\sA(F)${\it-semigroup with maximum }  $\sA(F)_{\rank}$ (referred to henceforth as $\M\sA(F)\rR$-semigroup).
\begin{proposition}\label{proposition51} Let $S$ be a $\MED$-semigroup and $F=\max(\msg(S)).$ Then $S\backslash \{F\}$ is a  $\M\sA(F)\rR$-semigroup. Moreover, every $\M\sA(F)\rR$-semigroup has this form.
\end{proposition}
\begin{proof}
	By Lemma \ref{lemma49}, we know that $S\backslash \{F\}\in \sA(F)$ and $\{x\in \msg(S\backslash \{F\})\mid x \notin \Delta(\sA(F))\}=\msg(S)\backslash \{F\}.$ By  Proposition \ref{proposition28},  we have $\sA(F)_{\rank}(S\backslash \{F\})=\m(S\backslash \{F\})-1.$
	
	If $S\in \sA(F)$ and $\sA(F)_{\rank}(S)=\m(S)-1,$ then by applying Proposition \ref{proposition28} we deduce that $S\cup \{F\}$ is a $\MED$-semigroup and $F=\max\left(\msg(S\cup \{F\}) \right).$
\end{proof}

If $A$ and $B$ are nonempty subsets of $\Z$, we write $A+B=\{a+b \mid a\in A, b\in B\}.$ 

The following result is deduced from \cite[Proposition I.2.9]{barucci}.

\begin{lemma}\label{lemma52}
	Let $S$ be a numerical semigroup. Then $S$ is a $\MED$-semigroup if and only if $(S\backslash \{0\})+\{-\m(S)\}$ is a numerical semigroup.	
\end{lemma}
The following result is deduced from \cite[Proposition 2]{belga}.

\begin{lemma}\label{lemma53} Let $S$ be a numerical semigroup and $m\in S\backslash \{0\}.$ Then $\left( \{m\}+S\right)\cup \{0\}$ is a $\MED$-semigroup with multiplicity $m$ and Frobenius number $\F(S)+m.$
	
\end{lemma}
\begin{proposition}\label{proposition54}
	Let $S$ be a numerical semigroup, $m\in S\backslash \{0\}$ and $F=\F(S)+2m.$ Then $T=\left( \{m\}+S\right)\cup \{0\}$ is a $\MED$-semigroup and $\max(\msg(T))=F.$ Moreover, if $P$ is a $\MED$-semigroup and $\max(\msg(P))=F,$ then $P$ has this form. 
\end{proposition}
\begin{proof}
	\begin{itemize}
	\item By Lemma \ref{lemma53}, we know that $T$ is a $\MED$-semigroup with multiplicity $m$ and Frobenius number $\F(S)+m.$ By using  Lemma \ref{lemma49}, we have that $\max(\msg(T))=\F(S)+m+m=F.$
	\item By Lemma \ref{lemma52}, we know that $S=(P\backslash \{0\})+\{-\m(P)\}$ is a numerical semigroup and $\m(P)\in S\backslash \{0\}.$ Moreover, by applying Lemma \ref{lemma49}, we have that $\F(S)=F-2\m(P).$ Finally, it is clear that $P=\left( \{\m(P)\}+S\right)\cup \{0\}.$	
	\end{itemize}
\end{proof}
\begin{example}\label{example}
	Let $S=\langle 5,7,9 \rangle=\{0,5,7,9,10,12,14,\rightarrow\}.$ Then, by applying Proposition \ref{proposition54}, we have that $T=\left( \{7\}+S\right)\cup \{0\}$ is a $\MED$-semigroup and $\max(\msg(T))=\F(S)+2\cdot 7=13+14=27.$ Therefore, by applying Proposition \ref{proposition51}, we obtain that $T\backslash \{27\}$ is a $\M\sA(F)\rR$-semigroup.
\end{example}

The following result is deduced from \cite[Corollary 8]{belga}.

\begin{lemma}\label{lemma56}
	Let $S$ be a $\MED$-semigroup. Then 
	$$
	 \g(S)=\displaystyle \frac{1}{\m(S)}\left(\sum_{x\in \msg(S)\backslash \{\m(S)\}}x\right)-\frac{\m(S)-1}{2}.
	$$	
\end{lemma}
\begin{proposition}\label{proposition57}
	Let $S$ be a $\M\sA(F)\rR$-numerical semigroup and let $A$ be its $\sA(F)$-minimal  system of generators. Then $$
	\g(S)=\displaystyle \frac{1}{\m(S)}\left(\sum_{a\in (A\backslash \{\m(S)\})\cup \{F\}}a\right)-\frac{\m(S)-3}{2}.
	$$
	
\end{proposition}
\begin{proof}
	From Proposition \ref{proposition51}, we can assert  that $S\cup \{F\}$ is a $\MED$-semigroup. By applying Proposition \ref{proposition28}, we have that $\msg(S\cup \{F\})=A\cup \{F\}.$ As $\g(S)=\g\left(S\cup \{F\}\right)+1,$ then by applying Lemma \ref{lemma56}, we obtain the desired  result. 
\end{proof}
\begin{example}
	Let $S=\sA(15)\left( \{8,9,10,11,12,13,14\}\right).$ Then, by applying Proposition \ref{proposition57}, we have that 
	$$
	\g(S)=\displaystyle \frac{1}{8}\left(9+10+11+12+13+14+15\right)-\frac{8-3}{2}=8.
	$$
	
\end{example}

\end{document}